\documentclass[12pt]{article}
 \usepackage{latexsym,amssymb,amsfonts,amsmath,amsthm,graphicx}
 \def\beq{ \begin{equation} }
 \def\eeq{ \end{equation} }
 \def\beqx{ \begin{equation*} }
 \def\eeqx{ \end{equation*} }
 \def\beqa{\begin{eqnarray}}
 \def\eeqa{\end{eqnarray}}
 \def\beqax{\begin{eqnarray*}}
 \def\eeqax{\end{eqnarray*}}
 \def\bal{\begin{allign}}
 \def\eal{\end{allign}}
 \def\balx{\begin{allign*}}
 \def\ealx{\end{allign*}}

 \def\mn{\medskip\noindent}
 \def\ms{\medskip}

 \def\ep{\epsilon}

\def\square{\vcenter{\vbox{\hrule height .4pt
  \hbox{\vrule width .4pt height 5pt \kern 5pt
        \vrule width .4pt} \hrule height .4pt}}}
\def\eopt{\hfill$\square$}
\def\var{\hbox{var\,}}

\numberwithin{equation}{section}
\newtheorem{theorem}{Theorem}
\newtheorem{lemma}{Lemma}[section]

\begin{document}

\title{Persistence of Activity in Threshold Contact Processes, \\
an ``Annealed Approximation" of Random Boolean Networks}
\author{Shirshendu Chatterjee and Rick Durrett
\thanks{Both authors were partially supported by NSF grant 0704996
from the probability program at NSF.} \\
Cornell University }
\date{\today}
\maketitle

 \begin{abstract}
 We consider a model for gene regulatory networks that is a modification of Kauffmann's (1969) random Boolean networks.
 There are three parameters: $n =$ the number of nodes, $r =$ the number of inputs to each node, and $p =$ the expected
 fraction of 1's in the Boolean functions at each node. Following a standard practice in the physics literature, we use
 a threshold contact process on a random graph on $n$ nodes, in which each node has in degree $r$, to approximate its dynamics.
 We show that if $r\ge 3$ and $r \cdot 2p(1-p)>1$, then the threshold contact process persists for a long time, which
 correspond to chaotic behavior of the Boolean network. Unfortunately, we are only able to prove the persistence time
 is $\ge \exp\left(cn^{b(p)}\right)$ with $b(p)>0$ when $r\cdot 2p(1-p)> 1$, and $b(p)=1$ when $(r-1)\cdot 2p(1-p)>1$.
 \end{abstract}

\vfill
\clearpage

 \section{Introduction}

 Random Boolean networks were originally developed by Kauffman (1969) as an abstraction of genetic regulatory networks.
 In our version of his model, the state of each node $x\in V_n \equiv \{1,2,\dots,n\}$ at time $t = 0, 1, 2, \ldots$ is
 $\eta_t(x)\in\{0,1\}$, and
 each node $x$ receives input from $r$ distinct nodes $y_1(x), \dots, y_r(x)$, which are chosen randomly from $V_n\setminus \{x\}$.

 We construct our random directed graph $G_n$ on the vertex set $V_n = \{ 1, 2, \ldots, n \}$ by putting oriented edges to each node from its input nodes.
 To be precise, we define
 the graph by creating a random mapping $\phi: V_n \times \{ 1, 2, \ldots, r \} \to V_n$, where $\phi(x,i) = y_i(x)$, such that $y_i(x) \neq x$ for $1\le i \le r$ and $y_i(x) \neq y_j(x)$ when $i \neq j$, and taking the edge set
 $E_n \equiv \{(y_i(x),x): 1\le i\le r, x \in V_n \}$. So each vertex has in-degree $r$ in our random graph $G_n$.
 The total number of choices for $\phi$ is $[(n-1)(n-2) \cdots (n-r)]^n$. However, the
 resulting  graph $G_n$ will remain the same under any permutation of
 the vector $\mathbf y_x\equiv (y_1(x), \ldots, y_r(x))$ for any
 $x\in V_n$. So if $e_{zx}\in \{0,1\}$ is the number of directed
 edges from node $z$ to node $x$ in $G_n$, then $\sum_{z=1}^n
 e_{z,x}=r$, and the total number of permutations of the vectors
 $\mathbf y_x, 1\le x\le n$, that correspond to the same graph is
 $(r!)^n$. So if $\mathbb{P}$ denotes the distribution of $G_n$, then
  \beqx \mathbb P (e_{zx}, 1\le z,
 x\le n)=\frac{(r!)^n}{[(n-1)(n-2)\cdots (n-r)]^n} =
 \frac{1}{\left[{n-1 \choose r}\right]^n},\eeqx
 if $e_{z,x}\in\{0,1\}, e_{x,x}=0$
 and $\sum_{z=1}^n e_{zx}=r$ for all $x\in V_n$,
  and $\mathbb P(e_{zx}, 1\le x, z\le n)=0$ otherwise.
 So our  random graph $G_n$ has uniform distribution over the collection of
 all directed graphs on the vertex set $V_n$ in which each vertex
 has in-degree $r$. Once chosen the network
 remains fixed through time. The rule for updating node $x$ is
 $$
 \eta_{t+1}(x) = f_x( \eta_t(y_1(x)),\ldots, \eta_t(y_r(x)) ),
 $$
 where the values $f_x(v)$, $x\in V_n$, $v\in \{ 0, 1\}^r$, chosen at the beginning and then fixed for all time, are independent and $=1$ with probability
 $p$.

A number of simulation studies have investigated the behavior of
this model. See Kadanoff, Coppersmith, and Aldana (2002) for survey.
Flyvberg and Kjaer (1988) have studied the degenerate case of $r=1$
in detail. Derrida and Pommeau (1986) have argued that for $r\geq 3$
there is a phase transition in the behavior of these networks
between rapid convergence to a fixed point and exponentially long
persistence of changes, and identified the phase transition curve to
be given by the equation $r \cdot 2p(1-p) =1$. The networks with
parameters below the curve have behavior that is `ordered', and
those with parameters above the curve have `chaotic' behavior. Since
chaos is not healthy for a biological network, it should not be
surprising that real biological networks avoid this phase. See
Kauffman (1993), Shmulevich, Kauffman, and Aldana (2005), and Nykter
et al.~(2008).

To explain the intuition behind the conclusion of Derrida and Pomeau (1986), we define another process $\{\zeta_t(x): t\ge 1\}$ for $x\in V_n$, which they called the {\it annealed approximation}. The idea is that $\zeta_t(x)=1$ if and only if $\eta_t(x) \neq
\eta_{t-1}(x)$, and $\zeta_t(x)=0$ otherwise. Now if the state of at
least one of the inputs $y_1(x), \dots, y_r(x)$ into node $x$ has
changed at time $t$, then the state of node $x$ at time $t+1$ will
be computed by looking at a different value of $f_x$. If we ignore
the fact that we may have used this entry before, we get the
dynamics of the threshold contact process
$$
P\left(\left. \zeta_{t+1}(x)=1 \right|
\zeta_t(y_1(x))+\cdots+\zeta_t(y_r(x)) > 0\right)= 2p(1-p),
$$
and $\zeta_{t+1}(x)=0$ otherwise. Conditional on the state at time
$t$, the decisions on the values of $\zeta_{t+1}(x)$, $x \in V_n$,
are made independently.

We content ourselves to work with the threshold contact process,
since it gives an approximate sense of the original model, and we
can prove rigorous results about its behavior. To simplify notation
and explore the full range of threshold contact processes we let $q
\equiv 2p(1-p)$, and suppose $0 \le q \le 1$. As mentioned above, it
is widely accepted that the condition for prolonged persistence of
the threshold contact process is $qr>1$. To explain this, we note
that vertices in the graph $G_n$ have average out-degree $r$, so a
value of 1 at a vertex will, on the average, produce $qr$ 1's in the
next generation.

We will also write the threshold contact process as a set valued
process. Let $\xi_t \equiv \{ x : \zeta_t(x) = 1 \}$. We will refer
to the vertices $x \in \xi_t$ as occupied at time $t$. So if $P_G$
is the distribution of the threshold contact process
$\boldsymbol{\xi} \equiv \{\xi_t: t\ge 0\}$ conditioned on the graph $G_n$,
then
 \beqax
 P_G\left(\left.x\in\xi_{t+1}\right|\{y_1(x), \ldots, y_r(x)\}\cap
 \xi_t\neq \emptyset\right) & = & q, \text{ and}\\
 P_G\left(\left.x\in\xi_{t+1}\right|\{y_1(x), \ldots, y_r(x)\}\cap
 \xi_t= \emptyset\right) & = & 0.
 \eeqax

Let $\boldsymbol \xi^A\equiv \left\{\xi_t^A: t\ge 0\right\}$
denote the threshold contact process starting from
$\xi_0^A=A\subset V_n$, and $\boldsymbol \xi^1\equiv \left\{\xi_t^1:
t\ge 0\right\}$ denote the special case when $A=V_n$. Let $\rho$ be
the survival probability of a branching process with offspring
distribution $p_r=q$ and $p_0=1-q$. By branching process theory \beq
\rho=1-\theta, \text{ where $\theta\in(0,1)$ satisfies } \theta = 1
- q + q \theta^r. \label{bpsurvp} \eeq

Using all the ingredients above we now present our first result.

\begin{theorem}\label{prolongpersist}
Suppose $q(r-1)>1$ and let $\delta>0$. Let $\mathbf P$ denote the
distribution of the threshold contact process $\boldsymbol \xi^1$,
starting from all sites occupied, on the random graph $G_n$, which
has distribution $\mathbb{P}$. Then there is a positive
constant $C(\delta)$ so that as $n \to \infty$
$$
\inf_{t\le \exp(C(\delta)n)} \mathbf P\left(\frac{|\xi_t^1|}{n}\ge
\rho-2\delta \right)\to 1.
$$
\end{theorem}

To prove this result, we will consider the dual coalescing branching
process $ \hat {\boldsymbol \xi}\equiv \{\hat\xi_t: t\ge 0\}$. In
this process if $x$ is occupied at time $t$, then with probability
$q$ all of the sites $y_1(x), \ldots, y_r(x)$ will be occupied at
time $t+1$, and with probability $1-q$ none of them will be occupied
at time $t+1$. Birth events from different sites are independent.
Let $\hat{\boldsymbol \xi}^A\equiv \{\hat \xi_t^A: t\ge 0\}$ be the
dual process starting from $\hat\xi_0^A=A\subset V_n$. The two
processes can be constructed on the same sample space so that for
any choices of $A$ and $B$ for the initial sets of occupied sites,
$\boldsymbol \xi^A$ and $\hat {\boldsymbol \xi}^B$ satisfies the
following duality relationship, see Griffeath (1978).
 \beq \left\{ \xi^A_t
 \cap B \neq \emptyset\right\} = \left\{ \hat \xi^B_t \cap A \neq
 \emptyset \right\}, \quad t=0, 1, 2, \ldots . \label{duality} \eeq

Taking $A = \{1, 2, \ldots, n \}$ and $B = \{x\}$ this says
\beq\label{duality1} \left\{ x \in \xi^1_t \right\} = \left\{ \hat
\xi^{\{x\}}_t \neq \emptyset \right\}, \eeq
 or, taking probabilities of both the events above, the density of
occupied sites in $\boldsymbol \xi^1$ at time $t$ is equal to the
probability that $\hat {\boldsymbol \xi}^{\{x\}}$ survives until
time $t$. Since over small distances our graph looks like a tree in
which each vertex has $r$ descendants, the last quantity $\approx
\rho$.

From \eqref{duality} it should be clear that we can prove Theorem
\ref{prolongpersist} by studying the coalescing branching process.
The key to this is an ``isoperimetric inequality". Let $\hat G_n$ be
the graph obtained from our original graph $G_n=(V_n,E_n)$ by
reversing the edges. That is, $\hat G_n = (V_n,\hat E_n)$, where
$\hat E_n = \{ (x,y) : (y,x) \in E_n \}$. Given a set $U \subset
V_n$, let \beq\label{ustar} U^* = \{ y \in V_n : x \to y \text{ for
some } x \in U \},\eeq where $x \to y$ means $(x,y) \in \hat E_n$.
Note that $U^*$ can contain vertices of $U$. The idea behind this
definition is that if $U$ is occupied at time $t$ in the coalescing
branching process, then the vertices in $U^*$ may be occupied at
time $t+1$.

\begin{theorem} \label{isoper}
Let $E(m,k)$ be the event that there is a subset $U\subset V_n$ with
size $|U|=m$ so that $|U^*| \le k$. Given $\eta>0$, there is an
$\ep_0(\eta)>0$ so that for $m \le \ep_0 n$
$$
\mathbb P\left[E(m,(r-1-\eta)m)\right] \le \exp(-\eta m \log(n/m)/2
).
$$
\end{theorem}

\noindent In words, the isoperimetric constant for small sets is
$r-1$. It is this result that forces us to assume $q(r-1) > 1$ in
Theorem \ref{prolongpersist}.

\mn {\bf Claim.} There is a $c>0$ so that if $n$ is large, then,
with high probability, for each $m \le cn$ there is  a set $U_m$
with $|U_m|=m$ and $|U_m^*| \le 1 + (r-1)m$.

\mn {\it Sketch of Proof.} Define an undirected graph $H_n$ on the
vertex set $V_n$ so that $x$ and $y$ are adjacent in $H_n$ if and
only if there is a $z$ so that $x \to z$ and $y \to z$ in $\hat
G_n$. The drawing illustrates the case $r=3$.

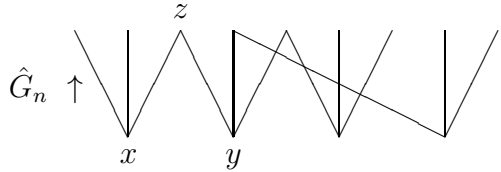
\begin{center}
\begin{picture}(240,100)
\put(70,40){\line(-1,2){20}}
\put(70,40){\line(0,1){40}}
\put(70,40){\line(1,2){20}}
\put(110,40){\line(-1,2){20}}
\put(110,40){\line(0,1){40}}
\put(110,40){\line(1,2){20}}
\put(150,40){\line(-1,2){20}}
\put(150,40){\line(0,1){40}}
\put(150,40){\line(1,2){20}}
\put(190,40){\line(-2,1){80}}
\put(190,40){\line(0,1){40}}
\put(190,40){\line(1,2){20}}
\put(67,30){$x$}
\put(107,30){$y$}
\put(87,85){$z$}
\put(47,55){$\uparrow$}
\put(25,55){$\hat G_n$}
\end{picture}
\end{center}

 \noindent The mean number of neighbors of a vertex in $H_n$ is $r^2
 \ge 9$, so standard arguments show that there is a $c>0$ so that,
 with probability tending to 1 as $n\to\infty$, there is a connected
 component $K_n$ of $H_n$ with $|K_n| \ge cn$. If $U$ is a connected
 subset of $K_n$ with $|U|=\lfloor cn\rfloor$, then by building up $U$ one vertex at
 a time and keeping it connected we get a sequence of sets $\{U_m,
 m=1, 2, \ldots, \lfloor cn\rfloor\}$ with $|U_m|=m$ and $|U_m^*| \le 1 + (r-1)m$.
 \eopt

\ms Since the isoperimetric constant is $\le r-1$, it follows that
when $q(r-1)< 1$, then for any $\ep>0$ there are bad sets $A$ with
$|A|\le n\ep$, so that $E\left|\hat\xi^A_1\right| \le |A|$.
Computations from the proof of Theorem \ref{isoper} suggest that
there are a large number of bad sets. We have no idea how to bound
the amount of time spent in bad sets, so we have to take a different
approach to show persistence when $1/r<q\le 1/(r-1)$.

\begin{theorem} \label{weakpersist}
Suppose $qr>1$. If $\delta_0$ is small enough, then for any
$0<\delta<\delta_0$, there are constants $C(\delta)>0$ and
$B(\delta)=(1/8-2\delta)\log(qr-\delta)/\log r$ so that as $n \to
\infty$
$$
\inf_{t\le \exp\left(C(\delta)\cdot n^{B(\delta)}\right)} \mathbf
P\left(\frac{|\xi_t^1|}{n}\ge \rho - 2 \delta \right)\to 1.
$$
\end{theorem}

 To prove this, we will again investigate persistence of the dual.
 Let
 \beqa
 d_0(x,y) & \equiv &\text{ length of a shortest oriented path from $x$
 to $y$ in $\hat G_n$}, \notag \\
 d(x,y) & \equiv & \min_{z\in V_n} [d_0(x,z) +
 d_0(y,z)],\label{dist} \eeqa
 and for any subset $A$ of vertices let \beq
 m(A,K)=\max_{S\subseteq A} \{|S|: d(x,y)\ge K \text{ for } x, y\in
 S, x\neq y\}. \label{mAdef} \eeq
 Let $R \equiv \log n/\log r$ be the
 average value of $d_0(1,y)$, let $a=1/8-\delta$ and $B=(a-\delta)\log
 (qr-\delta)/\log r$. We will show that if
 $m\left(\hat\xi^A_s,2\lceil aR\rceil\right) < \lfloor n^B\rfloor1$ at some time
 $s$, then with high probability, we will later have
 $m\left(\hat\xi^A_t,2 \lceil aR\rceil\right) \ge \lfloor n^B\rfloor$ for some
 $t>s$. To do this we explore the vertices in $\hat G_n$ one at
 a time using a breadth-first search algorithm based on the distance function $d_0$. We say that a
 collision has occurred if we encounter a vertex more than once in
 the exploration process. First we show in Lemma \ref{jointgrowth} that, with probability
 tending to 1 as $n \to \infty$,  there can be at most one collision
 in the set $\{u: d_0(x,u) \le 2\lceil aR\rceil\}$ for any $x\in
 V_n$. Then we argue in Lemma \ref{dualweakpersist}  that when we first have
 $m\left(\hat\xi^A_s,2\lceil aR\rceil\right) < \lfloor n^B\rfloor$, there is a subset $N$ of occupied sites
 so that $|N| \ge (q-\delta)\lfloor n^B\rfloor$, and
 $d(z,w)\ge 2\lceil aR\rceil-2$ for any two distinct vertices $z,w \in N$, and
 $\{u: d_0(z,u) \le 2\lceil aR\rceil -1\}$ has no collision. We run the
 dual process starting from the vertices of $N$ until time $\lceil
 aR\rceil-1$, so they are independent. With high probability there will be at
 least one vertex $w\in N$ for which $\left|\hat\xi_{\lceil aR\rceil-1}^{\{w\}}\right| \ge \lceil n^B\rceil$.
 By the choice of $N$, for any two distinct vertices $x, z\in
 \hat\xi_{\lceil aR\rceil-1}^{\{w\}}$, $d(x,z) \ge 2\lceil aR\rceil$.
 It seems foolish to pick only
 one vertex $w$, but we do not know how to
 guarantee that the vertices are suitably separated
 if we pick more.

\section{Proof of Theorem \ref{prolongpersist}}

We begin with the proof of the isoperimetric inequality, Theorem \ref{isoper}.

\begin{proof}[{\it Proof of Theorem 2}]
 Let $p(m,k)$ be the
probability that there is a set $U$ with $|U|=m$ and $|U^*|=k$.
First we will estimate $p(m,\ell)$ where
$\ell=\lfloor(r-1-\eta)m\rfloor$.
$$
 p(m,\ell) \le \sum_{\{(U,U'):|U|=m, |U'|=\ell\}} \mathbb P( U^* = U')
\le \sum_{\{(U,U'):|U|=m, |U'|=\ell\}} \mathbb P(U^* \subset U').
$$
According to the construction of $G_n$, for any $x\in U$ the other
ends of the $r$ edges coming out of it are distinct and they are
chosen at random from $V_n\setminus \{x\}$. So
$$
\mathbb P(U^* \subset U') = \left[\frac{{|U'|\choose r}}{{n-1
\choose r}} \right]^{|U|} \le \left( \frac{|U'|}{n-1}
\right)^{r|U|},
$$
and hence \beq p(m,\ell) \le \binom{n}{m} \binom{n}{\ell}
\left(\frac{\ell}{n-1}\right)^{rm}. \label{pml} \eeq

To bound the right-hand side, we use the trivial bound \beq
\binom{n}{m} \le \frac{n^m}{m!} \le \left(\frac{n e}{m}\right)^m,
\label{bcbdd} \eeq where the second inequality follows from $e^m >
m^m/m!$. Using (\ref{bcbdd}) in (\ref{pml})
$$
p(m,\ell) \le (ne/m)^{m} (ne/\ell)^{\ell}
\left(\frac{\ell}{n}\right)^{rm}\left(\frac{n}{n-1}\right)^{rm}.
$$
Recalling $\ell \le (r-1-\eta)m$, and accumulating the terms
involving $(m/n), r-1-\eta$ and $e$ the last expression becomes
 \beqax & \le & e^{m(r-\eta)} (m/n)^{m[-1-(r-1-\eta)+r]}
 (r-1-\eta)^{-(r-1-\eta)m+rm}[n/(n-1)]^{rm}\\
 & = & e^{m(r-\eta)} (m/n)^{m\eta} (r-1-\eta)^{m(1+\eta)}[n/(n-1)]^{rm}.
 \eeqax
 Letting $c(\eta)=r-\eta+ r\log(n/(n-1))+(1+\eta)\log(r-1-\eta)\le C$ for $\eta\in(0,r-1)$, we have
 $$
 p\left(m,\lfloor(r-1-\eta)m\rfloor\right) \le \exp\left(-\eta m\log(n/m) + Cm\right).
 $$
 Summing over integers $k = (r-1-\eta')m$ with $\eta' \ge \eta$, and
 noting that there are fewer than $rm$ terms in the sum, we have
 $$
\mathbb P\left[E(m,(r-1-\eta)m)\right] \le \exp(-\eta m\log(n/m) +
C'm).
$$

To clean up the result to the one given in Theorem \ref{isoper},
choose $\ep_0$ such that $\eta\log(1/\ep_0)/2>C'$. Hence for any
$m\le\ep_0n$,
$$
\eta\log(n/m)/2 \ge \eta\log(1/\ep_0)/2 >C',
$$
which gives the desired result.
\end{proof}

Our next goal is to show that the graph $\hat G_n$ locally looks
like a tree with high probability. For that we explore all the
vertices in $V_n$ one at a time, starting from a vertex $x$, and
using a breadth-first search algorithm based on the distance
function $d_0$ of \eqref{dist}. More precisely, for each $x\in V_n$,
we define the sets $A_x^k$, which we call the active set at the
$k^{th}$ step, and $R_x^k$, which we call the removed set at
$k^{th}$ step, for $k=0,1, \ldots, \beta_x$, where $\beta_x\equiv
\min\{l: A_x^l=\emptyset\}$, sequentially as follows.
$R_x^0\equiv\emptyset$ and $A_x^0\equiv\{x\}$. 
Let $D(x,l) = \{ y : d_0(x,y) \le l \}$. For $0\le k<\beta_x$,
we get $k_0=\min\{l: 0\le l\le k, A_x^k \cap D(x,l) \neq
\emptyset\}$, and choose $x_k \in A_x^k \cap D(x,k_0)$ with the
minimum index.
 \beqax
 &\text{ If } x_k \in R_x^k, & \text{ then }
 A_x^{k+1} \equiv A_x^k \setminus \{x_k\}, R_x^{k+1}\equiv R_x^k \text{ and }\\
 &\text{ if } x_k \not\in R_x^k, & \text{ then }
 A_x^{k+1}\equiv A_x^k \cup \left\{y_1(x_k), \ldots, y_r(x_k)\right\} \setminus \{x_k\},
 R_x^{k+1}\equiv R_x^k \cup \{x_k\}.
 \eeqax
 If $x_k\in R_x^k$, we say that a collision has occurred while exploring $\hat G_n$ starting from $x$.
 The choice of $x_k$ ensures that while exploring the graph
 starting from $x$, for any $j \ge 1$, we consider the vertices, which are at $d_0$
 distance $j$ from $x$, prior to those, which are at $d_0$ distance
 $j+1$ from $x$.

The next Lemma shows that with high probability $R_x^k$ will have
$k$ vertices, and for $x \neq z$, $R_x^k$ and $R_z^k$ do not
intersect each other, when $k \le n^{1/2-\delta}$. For the lemma we
need the following stopping times.
 \begin{align}\label{pi}
 & \pi_x^1 \equiv \min\left\{l\ge 1: |R_x^l|<l\right\}, \notag\\
 & \pi_{x,z} \equiv \min\left\{l\ge 1: R_x^l\cap R_z^l\neq \emptyset\right\}, x\neq z, \notag\\
 & \alpha_x^{n,\delta} \equiv \min\left\{l\ge 1: |R_x^l|\ge \lceil n^{1/2-\delta}\rceil\right\},
 \delta<1/2, \\
 & \beta_x=\min\left\{l\ge 1:A_x^l = \emptyset\right\}\notag
\end{align}
 So $\pi_x^1$ is the time of first collision while exploring $\hat
 G_n$ starting from $x$, and $\pi_{x,z}$ is the time of first
 collision while exploring $\hat G_n$ simultaneously from $x$ and
 $z$.
\begin{lemma}\label{growth}
Suppose $0<\delta<1/2$. Let $I_x^1$, $x\in V_n$, and $I_{x,z}$,
$x,z\in V_n, x\neq z$, be the events
 \beqx I_x^1 \equiv \left\{\pi_x^1 \wedge \beta_x \ge
\alpha_x^{n,\delta}\right\},\quad I_{x,z} \equiv I_x^1 \cap I_z^1
\cap \left\{\pi_{x,z} \ge \alpha_x^{n,\delta} \vee
\alpha_z^{n,\delta}\right\}, \eeqx where $\pi_x^1, \pi_{x,z},
\alpha_x^{n,\delta}$ and $\beta_x$ are the stopping times defined in
\eqref{pi}. Then
 \beq \mathbb P\left[\left(I_x^1\right)^c\right]\le
 n^{-2\delta},\quad \mathbb P(I_{x,z}^c) \le 5n^{-2\delta}
 \label{pnointer} \eeq for large enough $n$.
 \end{lemma}

 Note that the randomness, which determines whether the events $I_x^1$ and $I_{x,z}$ occur or not,
 arises only from the construction of the random graph $G_n$, and
 does not involve the threshold contact process $\boldsymbol{\xi}^1$
 on $G_n$.

 \begin{proof} Let $\delta'=1/2-\delta$. Since in the construction of the random graph $G_n$
 the input nodes $y_i(z), 1\le i\le r$, for any vertex $z$ are distinct and
 different from $z$, there are at least $n-r$ choices for each $y_i(z)$. Also  $\left| R_x^l\right| \le l$ for any
 $l$. So
 \beq\label{eq4} \mathbb P(|R_x^k|=|R_x^{k-1}|) \le (k-1)/(n-r).\eeq
 It is easy to check that $\pi_x^1\wedge \beta_x \ge \alpha_x^{n,\delta}$ if $|R_x^k| \neq |R_x^{k-1}|$ for $k=1, 2, \ldots, \lceil n^{\delta'} \rceil$. So
\begin{align*}
 \mathbb P\left[\left(I_x^1\right)^c\right]
 &\le  \mathbb P\left[ \cup_{k=1}^{\lceil n^{\delta'}\rceil } \left(\left|R_x^k\right|=\left|R_x^{k-1}\right|\right)\right]
  \le \sum_{k=1}^{\lceil n^{\delta'}\rceil} \mathbb  P\left(|R_x^k|=|R_x^{k-1}|\right) \\
 & \le \sum_{k=1}^{\lceil n^{\delta'}\rceil}  (k-1)/(n-r) \le n^{2\delta'}/n=n^{-2\delta}
\end{align*}
  for large enough $n$.
 For the other assertion, note that $I_{x,z}$ occurs if $|R_x^k| \neq
 |R_x^{k-1}|, |R_z^k| \neq |R_z^{k-1}|$ and $R_x^k \cap R_z^k =
 \emptyset$ for $k=1, 2, \ldots, \lceil n^{\delta'} \rceil$. Also if
 for some $k \ge 1$ $R_x^k \cap R_z^k \neq \emptyset$ and  $R_x^l \cap R_z^l = \emptyset$
 for all $1\le l <k$, then either $R_x^k=R_x^{k-1} \cup \{x_{k-1}\}$
 and $x_{k-1} \in R_z^{k-1}$, or $R_z^k=R_z^{k-1} \cup \{z_{k-1}\}$ and $z_{k-1} \in R_x^k$.
 Now since each of the input
 nodes in the construction of $G_n$ has at least $n-r$ choices,
 and $|R_x^l|, |R_z^l|\le l$ for any $l$,
 \beq\label{eq5} \mathbb P \left(R_x^k \cap R_z^k \neq \emptyset, R_x^l \cap R_z^l =
 \emptyset, 1\le l<k\right) \le \mathbb P\left(x_{k-1}\in R_z^{k-1}\right) + \mathbb P\left(z_{k-1} \in R_x^k\right)
 \le (2k-1)/(n-r).\eeq
 Combining the error probabilities of \eqref{eq4} and \eqref{eq5}
 \begin{align*} 
& \mathbb P \left(I_{x,z}^c\right)
  \le  \mathbb P \left[ \cup_{k=1}^{\lceil n^{\delta'} \rceil} \left(\left|R_x^k\right| =
 \left|R_x^{k-1}\right|\right)\cup_{k=1}^{\lceil n^{\delta'} \rceil} \left(\left|R_z^k\right| =
 \left|R_z^{k-1}\right|\right) \cup_{k=1}^{\lceil n^{\delta'} \rceil} \left(R_x^k
 \cap R_z^k \neq \emptyset\right)\right]\\
 & \le  \sum_{k=1}^{\lceil n^{\delta'}\rceil} \left[ \mathbb P\left(\left|R_x^k\right| =
 \left|R_x^{k-1}\right|\right) + \mathbb P\left(|R_z^k| = |R_z^{k-1}|\right)
 + \mathbb P \left(R_x^k \cap R_z^k \neq \emptyset, R_x^l \cap R_z^l =
 \emptyset, 1\le l<k\right)\right]\\
 & \le  \sum_{k=1}^{\lceil n^{\delta'}\rceil} (4k-3)/(n-r) \le
 5n^{2\delta'-1}=5n^{-2\delta}
  \end{align*}
  for large $n$.
 \end{proof}

Lemma \ref{growth} shows that $\hat G_n$ is locally tree-like. The
number of vertices in the induced subgraph $\hat G_{x,M}$ with vertex set $G_n \cap \{u: d_0(x,u) \le M\}$ 
is at most $1+r+\cdots +r^M \le 2r^M$. So if
$I_x^1$ occurs, then, for any $M$ satisfying $2r^M \le
n^{1/2-\delta}$, the subgraph $\hat G_{x,M}$ is an oriented finite $r-$tree, where each vertex
except the leaves has out-degree $r$. Similarly if $I_{x,z}$ occurs,
then for any such $M$, $\hat G_{x,M} \cap \hat G_{z,M} = \emptyset$.

In the next lemma, we will use this to get a bound on the survival
of the dual process for small times. Let $\rho$ be the branching
process survival probability defined in \eqref{bpsurvp}.

\begin{lemma}\label{startup}
If $q>1/r$, $\delta\in(0,qr-1)$, $\gamma=(20\log r)^{-1}$, and $b =
\gamma\log(qr-\delta)$ then for any $x\in V_n$, if $n$ is large,
$$
\mathbf P\left(\left|\hat\xi_{\lceil 2\gamma\log
n\rceil}^{\{x\}}\right|\ge \lceil n^b\rceil\right)\ge \rho-\delta.
$$
\end{lemma}

\begin{proof} Let $I_x^1$ be the event
$$
I_x^1=\left\{\pi_x^1 \wedge \beta_x \ge \alpha_x^{n,1/4}\right\},
$$
 where $\pi_x^1, \beta_x, \alpha_x^{n,1/4}$ are as in \eqref{pi}.
 Let $P_{Z^x}$ be the distribution of a branching process $\mathbf Z^x \equiv\{Z_t^x: t=0, 1, 2, \ldots\}$
 with $Z_0^x=1$ and offspring distribution $p_0=1-q$ and
$p_r=q$. Since $q>1/r$, this is a supercritical branching process.
Let $B_x$ be the event that the branching process survives. Then
 \beqx P_{Z^x}(B_x)= \rho,\eeqx
 where $\rho$ is as in \eqref{bpsurvp}. If we condition on $B_x$, then, using a large deviation result for
branching processes from Athreya (1994), \beq \label{brpldp}
P_{Z^x}\left(\left.\left|\frac{Z_{t+1}^x}{Z_t^x}-qr\right|>\delta
\right| B_x\right)\le e^{-c(\delta) t} \eeq for some constant
$c(\delta)>0$ and for large enough $t$. So if $F_x=\{ Z_{t+1}^x \ge
(qr-\delta)Z_t^x \hbox{ for } \lfloor \gamma \log n\rfloor \le t <
\lceil 2\gamma\log n\rceil \}$, then
 \beq \label{finq}
 P_{Z^x}(F_x^c|B_x)\le\sum_{t=\lfloor \gamma\log n\rfloor}^{(\lceil 2\gamma \log n\rceil)-1}
e^{-c(\delta) t}\le C_\delta n^{-c(\delta)\gamma/2} \eeq
 for some constant $C_\delta>0$ and for large
enough $n$. On the event $B_x\cap F_x$,
 $$Z_{\lceil 2\gamma\log n\rceil}^x\ge
 (qr-\delta)^{\lceil 2\gamma \log n\rceil -\lfloor \gamma\log
 n\rfloor} \ge (qr-\delta)^{\gamma\log n}=n^{\gamma\log(qr-\delta)},$$ since
$Z_{\lfloor \gamma \log n\rfloor}^x \ge 1$ on $B_x$.

Now coming back to the dual process $\hat{\boldsymbol \xi}^{\{x\}}$,
let $P_{I_x^1}$ denotes the conditional distribution of
$\hat{\boldsymbol{\xi}}^{\{x\}}$ given $I_x^1$. This
does not specify the entire graph but we will only use the conditional law for events
that involve the process on the subtree whose existence is guaranteed by $I_x^1$.
By the choice of
$\gamma$, the number of vertices in the subgraph induced by $\{u:
d_0(x,u) \le \lceil 2\gamma \log n\rceil\}$ is at most $2r^{\lceil
2\gamma \log n\rceil}<n^{1/4}$. Then it is easy to see that we can
couple $P_{I_x^1}$ with $P_{Z^x}$ so that
 \beqx P_{I_x^1}\left[\left(\left|\hat\xi_t^{\{x\}}\right|,0\le t\le \lceil 2\gamma\log n\rceil\right) \in
 \cdot\right]=P_{Z^x}\left[\left(Z_t^x,0\le t\le \lceil 2\gamma\log n\rceil\right) \in
 \cdot\right].\eeqx
 Combining the error probabilities of
\eqref{pnointer} and \eqref{finq}
 \beqax
 \mathbf P\left(\left|\hat\xi_{\lceil 2\gamma\log n\rceil}^{\{x\}}\right|\ge
 \lceil n^b\rceil\right)
 & \ge & P_{I_x^1}\left(\left|\hat\xi_{\lceil 2\gamma\log n\rceil}^{\{x\}}\right|\ge
 \lceil n^b\rceil\right) \mathbb P(I_x^1)\\
 & = & P_{Z^x}\left(Z_{\lceil 2\gamma \log n\rceil}^x\ge
 \lceil n^b\rceil\right)\mathbb P(I_x^1)\\
 & \ge & P_{Z^x}(B_x\cap F_x) \mathbb P(I_x^1)\\
 & = & P_{Z^x}(B_x) P_{Z^x}(F_x | B_x) \mathbb P(I_x^1)\\
 & \ge & \rho \left(1-C_\delta n^{-c(\delta)\gamma/2}\right) \left(1-n^{-1/2}\right) \ge
 \rho-\delta
 \eeqax
 for large enough $n$.
\end{proof}

Lemma \ref{startup} shows that the dual process starting from one
vertex will with probability $\ge\rho-\delta$ survive until there
are $\lceil n^b \rceil$ many occupied sites. The next lemma will
show that if the dual starts with $\lceil n^b \rceil$ many occupied
sites, then for some $\ep>0$ it will have $\lceil \ep n \rceil$ many
occupied sites with high probability.

\begin{lemma}\label{carryon}
If $q(r-1)>1$, then there exists $\ep_1 >0$ such that  for any $A$
with $|A| \ge \lceil n^b \rceil$ the dual process $\hat{\boldsymbol
\xi}^A$ satisfies
$$
\mathbf P\left(\max_{t\le \left\lceil \ep_1n-n^b
\right\rceil}\left|\hat\xi_t^A\right|<\ep_1
n\right)\le\exp\left(-n^{b/4}\right).
$$
\end{lemma}

\begin{proof} Choose $\eta>0$ such that $(q-\eta)(r-1-\eta)>1$, and let $\ep_0(\eta)$ be the constant in Theorem
\ref{isoper}.  Take $\ep_1\equiv \ep_0(\eta)$. Let $\nu \equiv
\min\left\{t: \left|\hat\xi_t^A\right|\ge \lceil\ep_1n
\rceil\right\}$. Let $F_t \equiv \left\{ \left|\hat\xi_t^A\right|\ge
\left|\hat\xi_{t-1}^A\right|+1\right\}$, and
\begin{align*}
B_t \equiv &\, \left\{ \text{at least $(q-\eta)\left|\hat\xi_t^A\right|$ occupied sites of $\hat\xi_t^A$ give birth} \right\},\\
C_t \equiv & \, \left\{ |U^*_t|\ge(r-1-\eta)|U_t| \right\}, \text{
where } U_t=\left\{x\in\hat\xi_t^A: x \text{ gives birth}\right\}.
\end{align*}
Now if $B_t$ and $C_t$ occur, then
 \beq\label{clarify} \left|\hat\xi_{t+1}^A\right|=|U_t^*|\ge (r-1-\eta)|U_t|\ge (r-1-\eta)(q-\eta)\left|\hat\xi_t^A\right|>\left|
 \hat\xi_t^A\right|,\eeq
 i.e. $F_{t+1}$ occurs. So $F_{t+1} \supseteq B_t \cap C_t$ for all $t\ge 0$. Using the binomial large deviations, see Lemma 2.3.3 on page
40 in Durrett (2007), \beq
P_G\left(\left.B_t\right|\hat\xi_t^A\right)\ge
1-\exp\left(-\Gamma((q-\eta)/q)q \left|\hat\xi_t^A\right|\right),
\label{bld} \eeq where $\Gamma(x)=x\log x-x+1>0$ for $x \neq 1$.
 If we take $H_0 \equiv \left\{\left|\hat\xi_0^A\right| \ge \lceil n^b\rceil\right\}$ and $H_t\equiv \cap_{s=1}^t F_s$, then $\left|\hat\xi_t^A\right| \ge
 \lceil n^b\rceil$ on the event $H_t$ for all $t\ge 0$. Keeping that in mind we can replace $\left|\hat\xi_t^A\right|$ in the right side of
 \eqref{bld} by $n^b$ to have
 \beq \label{bld1} P_G(B_t^c \cap H_t) \le P_G\left(B_t^c \cap \left\{\left|\hat\xi_t^A\right|
\ge \lceil n^b\rceil\right\}\right) \le
\exp\left(-\Gamma((q-\eta)/q)qn^b\right) \quad \forall t\ge 0.
 \eeq
 The same bound also works for the unconditional probability
 distribution $\mathbf P$. Next we see that
 $P_G(C_t | U_t) \ge \mathbf{1}_{E^c}$, where $E=E(|U_t|,
 (r-1-\eta)|U_t|)$, as defined in Theorem \ref{isoper}. Taking
 expectation with respect to the distribution of $G_n$, $\mathbf P(C_t | U_t) \ge \mathbb
 P(E^c)$. Since for $ t<\nu$, $|U_t|<\ep_0(\eta) n$, and
 $|U_t|\ge (q-\eta)n^b \ge n^b/(r-1)$ on $H_t \cap B_t$,
 using Theorem \ref{isoper}
 \begin{align}\label{bld2}
 \mathbf P(C_t^c \cap B_t \cap H_t \cap \{t<\nu\}) 
\le \mathbf P[C_t^c \cap \{(n^b/(r-1)) \le |U_t|< \ep_1 n\}] \notag\\
\le \exp\left(-\frac{\eta}{2}\frac{n^b}{r-1}\log \frac{n(r-1)}{n^b}\right).
 \end{align}
Combining these two bounds of \eqref{bld1} and \eqref{bld2} we get
 \beqax
 \mathbf P(F_{t+1}^c \cap H_t \cap \{t<\nu\})
 & \le & \mathbf P((B_t \cap C_t)^c \cap H_t \cap \{t<\nu\})\\
 & \le & \mathbf P(B_t^c \cap H_t) + \mathbf P(C_t^c \cap B_t \cap
 H_t \cap \{t<\nu\})
  \le \exp\left(-n^{b/2}\right)
 \eeqax
 for large $n$.
 Since $\nu\le \lceil\ep_1 n-n^b
 \rceil$ on $H_{\lceil \ep_1n-n^b \rceil}$,
 \beqax
 \mathbf P\left(\nu>\lceil \ep_1n-n^b \rceil\right)
 & \le & \mathbf P\left[ \left(\nu>\lceil \ep_1n-n^b \rceil\right)
 \cap \left(\cup_{t=1}^{\lceil \ep_1 n-n^b \rceil}
 F_t^c\right)\right]\\
 & \le & \sum_{t=1}^{\lceil \ep_1 n-n^b \rceil} \mathbf P(F_t^c\cap H_{t-1} \cap
 \{\nu>t-1\})\\
 & \le & (\lceil \ep_1 n-n^b\rceil)\exp\left(-n^{b/2}\right) \le
 \exp\left(-n^{b/4}\right)
 \eeqax
 for large $n$ and we get the result.
 \end{proof}

The next result shows that if there are $\lceil \ep n \rceil$ many
occupied sites at some time for some $\ep>0$, then the dual process
survives for at least $\exp(cn)$ units of time for some constant
$c$.

\begin{lemma}\label{dualpersist}
If $q(r-1)>1$, then there exist constants $c > 0$ and $\ep_1>0$ as
in Lemma \ref{carryon} such that for $T=\exp(cn)$ and any $A$ with
$|A|\ge \lceil \ep_1 n \rceil$,
$$
\mathbf P\left(\inf_{t\le
T}\left|\hat\xi_t^A\right|<\ep_1n\right)\le 2\exp(-cn).
$$
\end{lemma}

\begin{proof} Choose $\eta>0$ so that
$(q-\eta)(r-1-\eta)>1$, and then choose $\ep_0(\eta)>0$ as in
Theorem \ref{isoper}. Take $\ep_1=\ep_0(\eta)$. For any $A$ with
$|A|\ge \lceil \ep_1 n \rceil$, let $U'_t=\left\{x\in\hat\xi_t^A: x
\text{ gives birth}\right\}$, $t=0, 1, \ldots$. If $|U'_t|\le\lfloor
\ep_1 n\rfloor$, then take $U_t=U'_t$. If $|U'_t|>\ep_1 n$, we have
too many vertices to use Theorem \ref{isoper}, so we let $U_t$ be
the subset of $U'_t$ consisting of the $\lfloor \ep_1n \rfloor$
vertices with smallest indices. Let
\begin{align*}
 F_t= & \, \left\{ \left|\hat\xi_t^A\right| \ge \lceil \ep_1n \rceil \right\}, \qquad H_t = \cap_{s=0}^t F_s, \\
 B_t= & \, \left\{ \text{at least $(q-\eta)\left|\hat\xi_t^A\right|$ many occupied sites of $\hat\xi_t^A$ give birth} \right\}, \\
 C_t= & \, \{|U^*_t|\ge(r-1-\eta)|U_t| \}.
\end{align*}
Now using an argument similar for the one for \eqref{clarify},
$F_{t+1}\cap H_t \supset B_t\cap C_t \cap H_t$ for any $t \ge 0$.
Using our binomial large deviations result (\ref{bld}) again,
$P_G\left(\left. B_t \right| \hat\xi_t^A\right)\ge
1-\exp\left(-\Gamma((q-\eta)/q)q\left|\hat\xi_t^A\right|\right)$. On
the event $F_t$, $\left|\hat\xi_t^A\right|\ge \lceil\ep_1 n\rceil$,
and so
 $$P_G(B_t^c \cap H_t) \le P_G\left(B_t^c \cap \left\{\left|\hat\xi_t^A\right| \ge \lceil\ep_1n\rceil\right\}\right)
 \le \exp\left( -\Gamma((q-\eta)/q)q\ep_1 n\right).$$
 The same bound works for the unconditional probability distribution
 $\mathbf P$.

Since $|U_t|\le \ep_1n$, and on the event $H_t \cap B_t$ $|U_t|\ge
(q-\eta)\ep_1n \ge \ep_1n/(r-1)$, using Theorem \ref{isoper} and
similar argument which leads to \eqref{bld2} we have
$$
\mathbf P(C_t^c \cap H_t\cap B_t) \le
\exp\left(-\frac{\eta}{2}\frac{\ep_1 n}{r-1}\log
\frac{r-1}{\ep_1}\right).
$$
Combining these two bounds
\begin{align*}
\mathbf P(F_{t+1}^c \cap H_t) & \le \mathbf P[(B_t \cap C_t)^c \cap H_t] \\ 
& \le \mathbf P(B_t^c \cap H_t) + \mathbf P(C_t^c \cap B_t \cap H_t) \le 2\exp(-2c(\eta)n),
\end{align*}
 where
 $$
 c(\eta)=\frac{1}{2}\min\left\{\Gamma\left(\frac{q-\eta}{q}\right)q\ep_1,
 \frac{\eta}{2}\frac{\ep_1}{r-1}\log\frac{r-1}{\ep_1}\right\}.
 $$
 Hence for $T \equiv \exp(c(\eta)n)$
 \begin{align*}
 \mathbf P\left(\inf_{t\le T}\left|\hat\xi_t^A\right|<\ep_1n\right) 
& \le  \mathbf P\left(\cup_{t=1}^{\lfloor T \rfloor} F_t^c\right) \\
& \le
 \sum_{t=0}^{\lfloor T \rfloor-1} \mathbf P(F_{t+1}^c \cap G_t) \le
 2T\exp(-2c(\eta) n)=2\exp(-c(\eta) n).
 \end{align*}
 which completes the proof. \end{proof}

Lemma \ref{dualpersist} confirms prolonged persistence for the dual.
We will now give the

\mn {\it Proof of Theorem \ref{prolongpersist}.} Choose
$\delta\in(0,qr-1)$ and $\gamma=(20\log r)^{-1}$. Define the random
variables $Y_x, 1\le x\le n$, so that $Y_x=1$ if the dual process
$\hat{\boldsymbol \xi}^{\{x\}}$ starting at $x$ satisfies
$\left|\hat\xi_{\lceil 2\gamma\log n\rceil}^{\{x\}}\right|\ge \lceil
n^b\rceil$ for $b=\gamma\log(qr-\delta)$, and $Y_x=0$ otherwise. By
Lemma \ref{startup}, if $n$ is large, then 
$$
\mathbf EY_x\ge \rho-\delta\quad\hbox{for any $x$}.
$$

Let $\pi_x^1$, $\pi_{x,z}$ and $\alpha_x^{n,3/10}$ be the stopping
times as in \eqref{pi}, and $I_x^1,  I_{x,z}$ be the corresponding
events as in Lemma \ref{growth}. Recall that $\hat G_{x,M}$
is teh subgraph with vertex set $V_n \cap \{u: d_0(x,u) \le M\}$.
On the event $I_{x,z}$, $\hat G_{x, \lceil2\gamma\log n\rceil}$
and $\hat G_{z,\lceil2\gamma\log n\rceil}$ are oriented
finite $r-$trees consisting of disjoint sets of vertices, since
$2r^{\lceil 2\gamma\log n\rceil}\le n^{1/5}$ by the choice of
$\gamma$. Hence if $P_{I_{x,z}}$ is the conditional distribution of
$\left(\hat {\boldsymbol\xi}^{\{x\}}, \hat
{\boldsymbol\xi}^{\{z\}}\right)$ given $I_{x,z}$, then
 \beqax & &P_{I_{x,z}} \left[\left(\hat \xi_t^{\{x\}}, 0\le t\le \lceil 2\gamma\log n\rceil\right) \in \cdot ,
 \left(\hat \xi_t^{\{z\}}, 0\le t\le \lceil 2\gamma\log n\rceil\right) \in \cdot \right]\\
 & = &P_{I_{x,z}}\left[\left(\hat \xi_t^{\{x\}}, 0\le t\le \lceil 2\gamma\log n\rceil\right) \in \cdot\right]
 P_{I_{x,z}}\left[\left(\hat \xi_t^{\{z\}}, 0\le t\le \lceil 2\gamma\log n\rceil\right) \in
 \cdot\right].
 \eeqax

 Having all the ingredients ready we will now estimate the
 covariance between the events $\{Y_x=1\}$ and $\{Y_z=1\}$ for $x \neq z$. Standard
 probability arguments give the inequalities
 \beqax
 \mathbf P(Y_x=1, Y_z=1)
 &\le & \mathbf P[(Y_x=1, Y_z=1) \cap I_{x,z}] + \mathbb
 P(I_{x,z}^c)\\
 &= &  P_{I_{x,z}}(Y_x=1, Y_z=1) \mathbb P(I_{x,z})+\mathbb
 P(I_{x,z}^c)\\
 &= &  P_{I_{x,z}}(Y_x=1) P_{I_{x,z}}(Y_z=1) \mathbb P(I_{x,z})+\mathbb
 P(I_{x,z}^c)\\
 &= & \mathbf P[(Y_x=1) \cap I_{x,z}] \mathbf P[(Y_z=1) \cap
 I_{x,z}]/\mathbb P(I_{x,z}) +\mathbb
 P(I_{x,z}^c)\\
 &\le & \mathbf P(Y_x=1) \mathbf P(Y_z=1)/\mathbb P(I_{x,z}) + \mathbb
 P(I_{x,z}^c).
 \eeqax
 Subtracting $\mathbf P(Y_x=1) \mathbf P(Y_z=1)$ from both sides gives
 \beqa
 & & \mathbf P(Y_x=1, Y_z=1)-\mathbf P(Y_x=1) \mathbf P(Y_z=1) \notag\\
 &\le & \mathbf P(Y_x=1) \mathbf P(Y_z=1)\left(\frac{1}{\mathbb P(I_{x,z})}-1\right) + \mathbb
 P(I_{x,z}^c) \notag\\
 & \le & \mathbb P(I_{x,z}^c)[1+1/\mathbb P(I_{x,z})], \label{eq3}
 \eeqa
 where in the last inequality we replaced the two probabilities by 1.
 Now from Lemma \ref{growth}
 $\mathbb P(I_{x,z}^c)\le 5n^{-3/5}$, and so
 \beqx \mathbf P(Y_x=1, Y_z=1)-\mathbf P(Y_x=1) \mathbf P(Y_z=1)
 \le 5n^{-3/5}\left(1+1/\left(1-5n^{-3/5}\right)\right) \le 15n^{-3/5}\eeqx
 for large enough $n$. Using this bound,
$$
{\bf \var} \left(\sum_{x=1}^n Y_x\right)\le n+15 n(n-1) n^{-3/5},
$$
and Chebyshev's inequality shows that as $n\to\infty$
$$
\mathbf P\left(\left|\sum_{x=1}^n (Y_x-\mathbf EY_x)\right|\ge
n\delta\right) \le \frac{n+15n(n-1)n^{-3/5}}{n^2\delta^2}\to 0.
$$
Since $\mathbf EY_x\ge \rho-\delta$, this implies \beq\label{sum}
\lim_{n\to\infty} \mathbf P\left(\sum_{x=1}^n Y_x \ge
n(\rho-2\delta)\right)=1. \eeq

 Our next goal is to show that $\xi_T^1$ contains the random
 set $D \equiv \{x: Y_x=1\}$ at $T=T_1+T_2$, a time that grows exponentially fast in $n$.
 We choose $\eta>0$ so that $(q-\eta)(r-1-\eta)>1$. Let $\ep_1$ and
 $c(\eta)$ be the constants in Lemma \ref{dualpersist}. If
 $Y_x=1$, then $\left|\hat\xi_{T_1}^{\{x\}}\right| \ge \lceil n^b\rceil$ for $T_1=\lceil 2\gamma\log n\rceil$.
 Combining the error probabilities of Lemmas \ref{carryon} and
 \ref{dualpersist} shows that for $T_2=\left\lfloor\exp(c(\eta)n)\right\rfloor +
 \left\lceil \ep_1 n - n^b\right\rceil$, and for any subset $A$ of vertices with $|A|\ge \lceil n^b\rceil$
 \beq\label{eq1} \mathbf P\left(\left|\hat \xi_{T_2}^A\right| \ge \lceil\ep_1 n\rceil\right) \ge
 1-3\exp\left(-n^{b/4}\right)\eeq
 for large $n$.

 Let $\mathcal{C}$ be the set of all subsets of $V_n$ of size at
 least $\lceil n^b \rceil$, and denote $C_x \equiv \hat\xi_{T_1}^{\{x\}}$. Using
 the duality relationship of \eqref{duality1}
 for the conditional probability distribution
 \beqx \mathcal P(\cdot)= \mathbf P\left(\cdot \left| \hat \xi_t^{\{x\}}, 0\le
 t\le T_1, x\in V_n\right.\right),\eeqx
 we see that
 \beqax
 \mathcal P\left(\xi_{T_1+T_2}^1\supseteq D\right)
 & = & \mathcal P\left[\cap_{x\in D} \left(x\in \xi_{T_1+T_2}^1\right)\right]\\
 & = & \mathcal P\left[\cap_{x\in D} \left(\hat \xi_{T_1+T_2}^{\{x\}} \neq
 \emptyset\right)\right].
 \eeqax
 Since $D=\{x: Y_x=1\}$, it follows from the definition of $Y_x$ that
$C_x\in\mathcal C $ for all $x\in D$. So
 by the Markov property of the dual process the above is
 \beqax
 & = & \sum_{C_x \in \mathcal C, x\in D} \mathcal P\left[\cap_{x\in D} \left(\hat \xi_{T_1+T_2}^{\{x\}} \neq \emptyset, \hat \xi_{T_1}^{\{x\}}=C_x
 \right)\right]\\
 & = & \sum_{C_x \in \mathcal C, x\in D} \mathbf P\left[\cap_{x\in D} \left(\hat \xi_{T_2}^{C_x} \neq \emptyset\right)\right]
  \mathcal P\left[\cap_{x\in D} \left( \hat
 \xi_{T_1}^{\{x\}}=C_x\right)\right].
 \eeqax
 Using \eqref{eq1} $\mathbf P\left(\hat\xi_{T_2}^{C_x} \neq \emptyset\right)\ge
 \mathbf P\left(\left|\hat\xi_{T_2}^{C_x}\right|\ge \lceil\ep_1n\rceil\right)\ge 1-3\exp\left(-n^{b/4}\right)$. So the
 above is
 \beqax
 & \ge & \left(1-3|D|\exp\left(-n^{b/4}\right)\right) \sum_{C_x \in \mathcal C, x\in D} \mathcal P\left[\cap_{x\in D} \left( \hat \xi_{T_1}^{\{x\}}=C_x\right)\right]\\
 & \ge & 1-3n\exp\left(-n^{b/4}\right).
 \eeqax
 For the last inequality we use $|D|\le n$ and $\mathcal P(Y_x=1
 \forall x\in D)=1$. Since the lower bound only depends on $n$, the
 unconditional probability
 \beqx \mathbf P\left(\xi_{T_1+T_2}^1 \supseteq \{x: Y_x=1\}\right)
 \ge 1-3n\exp\left(-n^{b/4}\right).\eeqx

 Hence for $T=T_1+T_2$ using the attractiveness property of the
 threshold contact process, and combining the last calculation with
 \eqref{sum} we conclude that as $n\to\infty$
 \begin{align*} \inf_{t\le T} \mathbf P\left(\frac{|\xi^1_t|}{n}>\rho-2\delta\right)
 & = \mathbf P\left(\frac{|\xi^1_T|}{n}>\rho-2\delta\right) \\
 & \ge \mathbf P\left(\xi^1_T\supseteq \{x: Y_x=1\}, \sum_{x=1}^n Y_x\ge n
 (\rho-2\delta)\right)\to 1 .
 \end{align*}
 This completes the proof of Theorem \ref{prolongpersist}. \eopt

\section{Proof of Theorem \ref{weakpersist}}

 Recall the definition of the active sets $A_x^k, k=0, 1, \ldots, \beta_x$, and the removed
 sets $R_x^k, k=0, 1, \ldots, \beta_x$, introduced before Lemma \ref{growth}. Also recall the stopping times $\pi_x^1$ and
 $\alpha_x^{n,\delta}$ in \eqref{pi} and define
 \beqx \pi_x^2 \equiv \min \left\{l > \pi_x^1: \left| R_x^l \right|
 <l-1\right\}.\eeqx
 This is the time of second collision while exploring $\hat G_n$
 starting from $x$. First we show that with high probability for every vertex $x\in V_n$
 the second collision occurs after $\lceil n^{1/4-\delta} \rceil$ many steps for any
 $\delta\in(0,1/4)$.

 \begin{lemma}\label{jointgrowth}
 Let $\delta\in (0,1/4)$ and $I_x^2$ be the event
 \beqx I_x^2 \equiv \left\{\pi_x^2 \wedge \beta_x \ge
 \alpha_x^{n,1/4+\delta}\right\}.\eeqx
 Then for $I \equiv \cap_{x\in V_n} I_x^2$, $\mathbb P(I^c) \le 2n^{-4\delta}$ for large enough $n$.
 \end{lemma}

 \begin{proof}
 Let $\delta'=(1/4)-\delta$. Since in the construction of the random graph $G_n$
 the input nodes $y_i(z), 1\le i\le r$, for any vertex $z$ are distinct and
 different from $z$, there are at least $n-r$ choices for each $y_i(z)$. Also  $\left| R_x^l\right| \le l$ for any
 $l$. So $\mathbb P(|R_x^k|=|R_x^{k-1}|) \le (k-1)/(n-r)$. Now if $I_x^2$ fails to
 occur, then there will be $k_1$ and $k_2$ such that $1\le k_1<k_2
 \le \lceil n^{\delta'}\rceil$ and $|R_x^{k_i}|=|R_x^{k_i-1}|$ for
 $i=1, 2$. So
 \beqax \mathbb P\left[\left(I_x^2\right)^c\right]
 & \le & \sum_{1 \le
 k_1<k_2 \le \lceil n^{\delta'}\rceil} \mathbb P \left(\left|
 R_x^{k_1}\right|=\left| R_x^{k_1-1}\right| , \left|
 R_x^{k_2}\right|=\left| R_x^{k_2-1}\right|\right)\\
 & \le & \sum_{1 \le
 k_1<k_2 \le \lceil n^{\delta'}\rceil}
 \frac{(k_1-1)(k_2-1)}{(n-r)^2} \le \sum_{1\le k_1, k_2 \le \lceil n^{\delta'}\rceil}
 2\frac{(k_1-1)(k_2-1)}{n^2} \le 2n^{4\delta'-2}\eeqax
 for large enough $n$. The second inequality holds because the choices of the input nodes are independent.
 Hence $\mathbb P(I^c) \le \sum_{x\in V_n} \mathbb P
 \left[\left(I_x^2\right)^c\right] \le 2n^{4\delta'-1}=2n^{-4\delta}$.
  \end{proof}

  Lemma \ref{jointgrowth} shows that with high probability for all vertices there will 
  be at most one collision until we have explored $\lceil
  n^{1/4-\delta}\rceil$ many vertices starting from any vertex of
  $\hat G_n$. Now recall the definition of the distance functions $d_0$ and $d$ from \eqref{dist}, and $m(A,K)$ given in \eqref{mAdef}. Let
 $R=\log n/\log r$, $a=(1/8-\delta)$ and let $\rho$ be the branching
 process survival probability defined in \eqref{bpsurvp}.

 \begin{lemma}\label{dualweakpersist}
 Let $P_I$ denote the conditional distribution of $\hat{\boldsymbol\xi}^{\{x\}}, x\in V_n$
 given $I$, where $I$ is the event defined in Lemma \ref{jointgrowth}. If $qr>1$ and $\delta_0$
 is small enough, then for any
 $0<\delta<\delta_0$ there are constants $C(\delta)>0$,
 $B(\delta)=(1/8-2\delta)\log(qr-\delta)/\log r$ and a stopping time
 $T$ satisfying
 $$
 P_I\left(T<2\exp\left(C(\delta)n^{B(\delta)}\right)\right)\le
 2\exp\left[-C(\delta) n^{B(\delta)}\right],
 $$
such that for any $A$ with $m(A, 2\lceil aR\rceil)\ge \lfloor
n^{B(\delta)}\rfloor$, $\left|\hat\xi_T^A\right|\ge \lfloor
n^{B(\delta)}\rfloor$.
\end{lemma}

 \begin{proof} Let $m_t \equiv m\left(\hat\xi_t^A,2\lceil aR\rceil\right)$. We define the
 stopping times $\sigma_i$ and $\tau_i$ as follows. $\sigma_0 \equiv
 0$, and for $i\ge 0$
 \begin{align*}
 \tau_{i+1} & \equiv \min\left\{t>\sigma_i: m_t< \lfloor n^B\rfloor\right\},\\
 \sigma_{i+1} & \equiv \min\left\{t>\tau_{i+1}: m_t\ge \lfloor  n^B\rfloor\right\}.
 \end{align*}

 Since $\tau_i>\sigma_{i-1}$ for $i\ge 1$,
 $m_{\tau_i-1}\ge \lfloor n^B\rfloor$, and hence there is a set $X_i
 \subset \hat \xi^A_{\tau_i-1}$ of size at least $\lfloor n^B\rfloor$
 such that $d(u,v)\ge 2\lceil aR\rceil$ for any two distinct vertices $u, v\in X_i$.
 Let $E_i$ be the event that at least $(q-\delta)|X_i|$ many
 vertices of $X_i$ give birth at time $\tau_i$.
 Using the binomial large deviation estimate \eqref{bld}
 \beq\label{eq6} P_G(E_i) \ge 1-\exp\left(-\Gamma((q-\delta)/q)q\lfloor n^B\rfloor\right),\eeq
 where $\Gamma(x)=x\log x -x +1$.

 Now let $I$ be the event defined in Lemma \ref{jointgrowth}. Since
 $\left|\left\{z: d_0(x,z) \le 2\lceil aR \rceil\right\}\right|$ is at most $2r^{2\lceil
 aR \rceil}\le 2rn^{2a} \le n^{1/4-\delta}$, so if $I$ occurs, then
 for any vertex $x\in V_n$ there is at most one collision in $\left\{z: d_0(x,z) \le 2\lceil aR
 \rceil\right\}$, and hence there are at least $r-1$ input nodes $u_1(x),
 \ldots, u_{r-1}(x)$ of $x$ such that $\left\{z: d_0(u_i(x),z) \le 2\lceil aR
 \rceil -1\right\}$ is a finite oriented $r-$tree for each $1\le i\le
 r-1$. Since the right side of \ref{eq6} depends only on $n$,
 \beqx P_I(I\cap E_i) = P_I(E_i) \ge 1-\exp\left(-c_1(\delta)n^B\right), \eeqx
 where $c_1(\delta)=\Gamma((q-\delta)/q)q/2$. If $I \cap E_i$
 occurs, then we can choose one suitable offspring
 of each of the vertices in $X_i$, which give birth, to form a subset
 $N_i \subset \hat\xi_{\tau_i}^A$ such that $|N_i| \ge (q-\delta)
 \lfloor n^B\rfloor$, $d(u,v) \ge 2\lceil aR\rceil-2$ for any two distinct vertices $u,v \in
 N_i$, and $\{z: d_0(u,z) \le 2\lceil aR
 \rceil -1\}$ is a finite oriented $r-$tree for each $u \in N_i$.

 By the definition of $N_i$ it is easy to see that for each $x\in
 N_i$
 \beqx P_I \left[\left(\left|\hat\xi_t^{\{x\}}\right|,
 0\le t\le 2\lceil aR\rceil-1 \right)\in \cdot\right]
 =P_{Z^x}\left[\left(Z_t^x, 0\le t\le 2\lceil aR\rceil-1 \right) \in \cdot \right],\eeqx
 where $\mathbf Z^x$ is a supercritical branching process, as introduced in Lemma \ref{startup},  with distribution $P_{Z^x}$ and mean
 offspring number $qr$. Let $B_x$ be the event of survival for
 $\mathbf Z^x$, and $F_x=\cap_{t=\lfloor \delta R\rfloor-1}^{\lceil
 aR\rceil-2} \left\{ Z_{t+1}^x\ge(qr-\delta)Z_t^x\right\}$. So
 $P_{Z^x}(B_x)=\rho>0$ as in \eqref{bpsurvp}. Using the error
 probability of \eqref{brpldp}
  \beq\label{eq2} P_{Z^x}(F_x^c|B_x)\le
  \sum_{t=\lfloor \delta R\rfloor-1}^{\lceil aR\rceil-2} e^{-c'(\delta)t} \le C_\delta
  e^{-c'(\delta)\delta\log n/(2\log r)}= C_\delta
  n^{-c'(\delta)\delta/(2\log r)} \eeq
  for some constants $C_\delta, c'(\delta)>0$. On the event $B_x\cap F_x$,
 $$Z_{\lceil aR \rceil-1}^x\ge
 (qr-\delta)^{(\lceil aR\rceil-1)-(\lfloor \delta R\rfloor-1)} \ge (qr-\delta)^{(a-\delta)R}=n^{(a-\delta)\log(qr-\delta)/\log r}=n^B.$$
 Hence for $Q_x \equiv \left\{ \left|\hat\xi_{\lceil aR\rceil-1}^{\{x\}}\right| \ge \lceil
 n^B\rceil \right\}$ for $x\in N_i$, we use standard probability arguments and  \eqref{eq2}  to
 have
 \beqa
 P_I(Q_x) & = & P_I\left(\left|\hat\xi_{\lceil aR\rceil-1}^{\{x\}}\right| \ge \lceil n^B\rceil\right) 
  =  P_{Z^x}\left(Z_{\lceil aR\rceil-1}^x \ge \lceil n^B\rceil\right) \notag\\
 & \ge & P_{Z^x}(B_x \cap  F_x)
 \ge P_{Z^x}(B_x ) P_{Z^x}(F_x | B_x)
 \ge \rho-\delta \label{survive}\eeqa
 for large enough $n$.

 Since $d(u,v)\ge 2\lceil aR\rceil-2$ for any two distinct vertices $u, v\in N_i$,
 $\hat\xi_t^{N_i}$ is a disjoint union of $\hat\xi_t^{\{x\}}$ over
 $x\in N_i$ for $t\le \lceil aR\rceil-1$. Let $H_i$ be the event
 that there is at least one $x\in N_i$ for which $Q_x$ occurs. Then
  recalling that $|N_i| \ge (q-\delta)\lfloor n^B\rfloor$ on
  $E_i$,
 \beq P_I(H_i^c|E_i) \le
 (1-\rho+\delta)^{(q-\delta)
 \lfloor n^B\rfloor}=\exp\left(-c_2(\delta)n^B\right), \eeq
 where $c_2(\delta)=(q-\delta)\log(1/(1-\rho+\delta))/2$.

 If $H_i\cap E_i$ occurs, choose any vertex $w_i\in N_i$ such that
 $Q_{w_i}$ occurs and let $S_i \equiv \hat\xi_{\lceil aR\rceil-1}^{\{w_i\}}$. By
 the choice of $w_i$, $|S_i|\ge \lfloor n^B\rfloor$. Since $(\lceil aR\rceil-1) +\lceil aR\rceil
 = 2\lceil aR\rceil -1$, for any two distinct vertices $x,z\in S_i$  the subgraphs induced by
 $\left\{u: d_0(x,u)\le \lceil aR \rceil\right\}$ and $\left\{u: d_0(z,u)\le \lceil aR
 \rceil\right\}$ are finite $r-$trees consisting of disjoint sets of
 vertices, and hence $d(x,z) \ge 2\lceil aR\rceil$. Hence using monotonicity of the
 dual process $\sigma_i\le \tau_i+\lceil aR\rceil-1$ on this event $H_i \cap E_i$. So
 $$
 P_I(\sigma_i>\tau_i+\lceil aR\rceil-1)\le P_I(E_i^c)+ P_I(H_i^c|E_i) \le 2\exp(-2C(\delta)n^B),
 $$ where $C(\delta) \equiv \min\{c_1(\delta), c_2(\delta)\}/2$. Let $L=\inf\{i\ge 1: \sigma_i>\tau_i+\lceil aR\rceil-1\}$. Then
 \begin{align*}
 P_I\left[L>\exp\left(C(\delta)n^B\right)\right] & \ge \left[1-2\exp(-2C(\delta)n^B)\right]^{\exp\left(C(\delta)n^B\right)}\\
& \ge 1-2\exp\left(-C(\delta)n^B\right).
\end{align*}
Since $\sigma_i>\tau_i>\sigma_{i-1}$, $\sigma_{L-1}\ge 2(L-1)$. As
$\left|\hat\xi_{\sigma_{L-1}}^A\right|\ge  \lfloor n^B\rfloor$, we
get our result if we take $T=\sigma_{L-1}$.
 \end{proof}

As in the proof of Theorem \ref{prolongpersist}, survival of the
dual process gives persistence of the threshold contact process.

 \mn {\it Proof of Theorem \ref{weakpersist}.} Let
 $0<\delta<\delta_0, \rho$, $a=(1/8-\delta)$ and
 $B=(1/8-2\delta)\log(qr-\delta)/\log r$ be the constants from the
 previous proof. Define the random variables $Y_x, 1\le x\le n$, as
 $Y_x=1$ if the dual process $\hat{\boldsymbol \xi}^{\{x\}}$ starting
 at $x$ satisfies $\left|\hat\xi_{\lceil
 aR\rceil-1}^{\{x\}}\right|> \lfloor n^B\rfloor$ and $Y_x=0$
 otherwise.

 Consider the event $I_x^1=\left\{\pi_x^1 \wedge \beta_x \ge
 \alpha_x^{n,1/4+\delta}\right\}$, where $\pi_x^1, \beta_x$ and
 $\alpha_x^{n,1/4+\delta}$ are stopping times defined as in \eqref{pi}.
 Using Lemma \ref{growth} and \ref{jointgrowth}
 \beq\label{eq7} P_I\left[\left(I_x^1\right)^c \right]
 \le \frac{\mathbb P\left[\left(I_x^1\right)^c\right]}{\mathbb P(I)}
 \le \frac{n^{-2(1/4+\delta)}}{1-2n^{-4\delta}} \le
 2n^{-(1/2+2\delta)}.\eeq
 Let $J_x \equiv I \cap I_x^1$ and $P_{J_x}$ be the conditional distribution of
 $\hat{\boldsymbol\xi}^{\{x\}}$ given $J_x$. Since the number of vertices in the set $\{u: d_0(x,u) \le \lceil
 aR\rceil-1\}$ is at most $2r^{\lceil aR\rceil-1} \le
 2r^{aR}<n^{1/4-\delta}$ by the choice of $a$,
 \beqx P_{J_x} \left[\left(\left|\hat\xi_t^{\{x\}}\right|,
 0\le t\le \lceil aR\rceil-1 \right)\in \cdot\right]
 =P_{Z^x}\left[\left(Z_t^x, 0\le t\le \lceil aR\rceil-1 \right) \in \cdot \right],\eeqx
 where $\mathbf Z^x$ is a supercritical branching process, as introduced in Lemma \ref{startup},  with distribution $P_{Z^x}$ and mean
 offspring number $qr$. Let $B_x$ and $F_x=\cap_{t=\lfloor \delta R\rfloor-2}^{\lceil aR\rceil-2} \left\{Z_{t+1}^x \ge
 (qr-\delta)Z_t^x\right\}$. So $P_{Z^x}(B_x)=\rho>0$ as in
 \eqref{bpsurvp}, and similar to \eqref{eq2}
 \beqx P_{Z^x} (F_x^c | B_x) \le \sum_{t=\lfloor \delta R\rfloor-2}^{\lceil
 aR\rceil-2} e^{-c'(\delta) t} \le C_\delta
 n^{-c'(\delta)\delta/(2\log r)}\eeqx
 for some constants $C_\delta, c'(\delta)>0$. On the event $B_x \cap
 F_x$, $Z_{\lceil aR\rceil-1}^x \ge (qr-\delta)^{(\lceil
 aR\rceil-1)-(\lfloor \delta
 R\rfloor-2)}>(qr-\delta)^{(a-\delta)R} \ge \lfloor n^B\rfloor$. Hence using
 \eqref{eq7}
 \beqax
 P_I(Y_x=1)
 & \ge &  P_I\left(I_x^1 \cap \left\{\left|\hat\xi_{\lceil aR\rceil -1}^{\{x\}}\right|
 > \lfloor n^B\rfloor \right\}\right)\\
 & = &  P_{J_x}\left(\left|\hat\xi_{\lceil aR\rceil -1}^{\{x\}}\right|
 > \lfloor n^B\rfloor\right) P_I(I_x^1)\\
 & = & P_{Z^x}\left(Z_{\lceil aR\rceil-1}^x > \lfloor n^B\rfloor\right)
 P_I(I_x^1)\\
 & \ge & P_{Z^x}(B_x \cap F_x) P_I(I_x^1)
 =P_{Z^x}(B_x) P_{Z^x} (F_x | B_x) P_I(I_x^1) \ge
 \rho-\delta
 \eeqax
 for large enough $n$.

 Next we estimate the covariance between the events $\{Y_x=1\}$ and
 $\{Y_z=1\}$. We consider the stopping times
 $\pi_x^1, \beta_x, \pi_{x,z}, \alpha_x^{n,1/4+\delta}$ as in \eqref{pi} and the
 corresponding event $I_{x,z}$ as in Lemma \ref{growth}. We can use similar argument,
 which leads to \eqref{eq3}, to conclude
 \beqx P_I(Y_x=1,Y_z=1)-P_I(Y_x=1) P_I(Y_z=1) \le
 P_I(I_{x,z}^c)(1+1/P_I(I_{x,z})).\eeqx
 From Lemma \ref{growth} and \ref{jointgrowth},
 \beqx P_I(I_{x,z}^c) \le \frac{\mathbb P(I_{x,z}^c)}{\mathbb P(I)} \le \frac{5n^{-2(1/4+\delta)}}
 {1-2N^{-4\delta}} \le  10n^{-(1/2+2\delta)} \eeqx
 for large enough $n$, and so
 \beqx P_I(Y_x=1,Y_z=1)-P_I(Y_x=1) P_I(Y_z=1) \le 30 n^{-(1/2+2\delta)}\eeqx
 for large $n$. Using the bound on the covariances,
 $$
 \hbox{var}_I\left(\sum_{x=1}^n Y_x\right)\le n +30 n(n-1) n^{-2\delta},
 $$
 and Chebyshev's inequality gives that as $n\to\infty$
 $$
 P_I\left(\left|\sum_{x=1}^n (Y_x-\mathbf EY_x)\right|\ge n\delta\right) \le \frac{n+30 n(n-1)n^{-2\delta}}{n^2\delta^2}\to 0.
 $$
 Since $\mathbf EY_x\ge \rho-\delta$ for all $x\in V_n$, this implies
\beq\label{sum2} \lim_{n\to\infty}  P_I\left(\sum_{x=1}^n Y_x \ge
n(\rho-2\delta) \right)=1. \eeq

 Our next goal is to show that $\xi_T^1$ contains the random
 set $D \equiv \{x: Y_x=1\}$ with high probability for a suitable choice of $T$.
 If $Y_x=1$, then $\left|\hat\xi_{T_1}^{\{x\}}\right| > \lfloor n^B\rfloor$,
 where $T_1=\lceil aR\rceil -1$. Note that $\lceil
 aR\rceil-1+\lceil aR\rceil \le 2\lceil aR\rceil$, and on the event $I$ there can be at most one
 collision in $\{u: d_0(x,u) \le 2\lceil aR\rceil\}$. Even though the first collision occurs between descendants
 of two vertices in $\hat\xi_{T_1}^{\{x\}}$,  still we
 can exclude one vertex from $\hat\xi_{T_1}^{\{x\}}$ to have a set $W_x\subset
 \hat\xi_{T_1}^{\{x\}}$ of size at least $\lfloor n^B\rfloor$ such that for any two distinct vertices
 $z, w\in W_x$, the subgraphs induced by $\{u: d_0(z,u) \le \lceil
 aR\rceil\}$ and $\{v: d_0(w,v) \le \lceil aR\rceil\}$ are finite
 oriented $r-$trees consisting of disjoint sets of vertices, i.e. $d(z,w) \ge 2\lceil aR\rceil$.
 So if $Y_x=1$, then  $m\left(\hat\xi_{T_1}^{\{x\}},2\lceil aR\rceil\right)\ge
 \lfloor n^B\rfloor$ on the event $I$. Using Lemma \ref{dualweakpersist}, after an
 additional $T _2\ge 2\exp\left(C(\delta)n^B\right)$ units of time, the dual
 process contains at least $\lfloor n^B\rfloor$ many occupied sites with $P_I$
 probability $\ge 1-2\exp\left(-C(\delta)n^B\right)$.

 Let $\mathcal{F}$ be the set of all subsets of $V_n$ of size
 $>\lfloor n^B\rfloor$, and denote $F_x \equiv \hat\xi_{T_1}^{\{x\}}$. Using
 the duality relationship of \eqref{duality1}
 for the conditional probability $\mathcal P_I(\cdot) \equiv
 \mathcal P(\cdot|I)$, where
 \beqx \mathcal P(\cdot)= \mathbf P\left(\cdot \left| \hat \xi_t^{\{x\}}, 0\le
 t\le T_1, x\in V_n\right.\right),\eeqx
 we see that
 \beqax
 \mathcal P_I\left(\xi_{T_1+T_2}^1\supseteq D\right)
 & = & \mathcal P_I\left[\cap_{x\in D} \left(x\in \xi_{T_1+T_2}^1\right)\right]\\
 & = & \mathcal P_I\left[\cap_{x\in D} \left(\hat \xi_{T_1+T_2}^{\{x\}} \neq
 \emptyset\right)\right].
 \eeqax
 Since $D=\{x: Y_x=1\}$, $F_x\in\mathcal F$ for all $x\in D$. So
 by the Markov property of the dual process the above is
 \beqax
 & = & \sum_{F_x \in \mathcal F, x\in D} \mathcal P_I\left[\cap_{x\in D} \left(\hat \xi_{T_1+T_2}^{\{x\}} \neq \emptyset, \hat \xi_{T_1}^{\{x\}}=F_x
 \right)\right]\\
 & = & \sum_{F_x \in \mathcal F, x\in D} P_I\left[\cap_{x\in D} \left(\hat \xi_{T_2}^{F_x} \neq \emptyset\right)\right]
  \mathcal P_I\left[\cap_{x\in D} \left( \hat
 \xi_{T_1}^{\{x\}}=F_x\right)\right].
 \eeqax
 Now since $W_x\subset F_x$,  using monotonicity of the dual process, $P_I\left(\hat\xi_{T_2}^{F_x} \neq \emptyset\right)\ge
 P_I\left(\hat\xi_{T_2}^{W_x} \neq \emptyset\right)$. Also using
 Lemma \ref{dualweakpersist}, $P_I\left(\left|\hat\xi_{T_2}^{W_x}\right|\ge \lfloor n^B\rfloor\right) \ge 1-2\exp\left(-C(\delta)n^B\right)$ for any $F_x \in \mathcal F$. So the
 above is
 \beqax
 & \ge & \left(1-2|D|\exp\left(-C(\delta)n^B\right)\right) \sum_{F_x \in \mathcal F, x\in D} \mathcal P_I\left[\cap_{x\in D} \left( \hat \xi_{T_1}^{\{x\}}=F_x\right)\right]\\
 & \ge & 1-2n\exp\left(-C(\delta)n^B\right).
 \eeqax
 For the last inequality we use $|D|\le n$ and $\mathcal P_I(Y_x=1
 \forall x\in D)=1$. Since the lower bound only depends on $n$,
 \beqax P_I\left(\xi_{T_1+T_2}^1 \supseteq \{x: Y_x=1\}\right)
 & \ge & 1-2n\exp\left(-C(\delta)n^B\right) \\
 \Rightarrow \quad  \mathbf P\left(\xi_{T_1+T_2}^1 \supseteq \{x: Y_x=1\}\right)
 & \ge & \mathbb P(I) \left[1-3n\exp\left(-C(\delta)n^B\right)\right] \to 1,\eeqax
 as $n\to \infty$, since $\mathbb P(I) \ge 1-2n^{-4\delta}$ by Lemma \ref{jointgrowth}.

 Hence for $T=T_1+T_2$ using the attractiveness
 property of the threshold contact process, and combining the last
 calculation with \eqref{sum2} we conclude that as $n\to\infty$
 \begin{align*}
 \inf_{t\le T} \mathbf P\left(\frac{|\xi^1_t|}{n}>\rho-2\delta\right)
 & =\mathbf P\left(\frac{|\xi^1_T|}{n}>\rho-2\delta\right) \\
 & \ge \mathbf P\left(\xi^1_T\supseteq \{x: Y_x=1\}, \sum_{x=1}^n Y_x\ge
 n(\rho-2\delta)\right)\to 1,
 \end{align*}
 which completes the proof of Theorem \ref{weakpersist}. \eopt

\section*{References}

\mn
Albert, R. and Othmer, H. G. (2003) The topology of the regulatory interactions predicts the expression pattern of the segment polarity genes in {\it Drosophila melanogaster}. {\it Journal of Theoretical Biology}, 223, 1--18

\mn
Athreya, K.B. (1994) Large deviations for branching processes, I. Single type case.
{\it Ann. Appl.  Prob.} 4, 779--790

\mn
Chaves, M., Albert, R., and Sontag, E.D. (2005) Robustness and fragility of Boolean models for genetic regulatory networks.
{\it J. Theor. Biol.} 235, 431--449

\mn
Derrida, B. and Pomeau, Y. (1986) Random networks of automata: a simplified annealed approximation. {\it Europhysics Letters}, 1, 45--49

\mn Durrett, R. (2007) {\it Random Graph Dynamics.} Cambridge
University Press.

\mn
Flyvbjerg, H. and Kjaer, N. J. (1988) Exact solution of Kaufmann's model with connectivity one. {\it Journal of Physics A}, 21, 1695--1718

\mn
Griffeath, D. (1978) {\it Additive and cancellative interacting particle systems.}
Lecture Notes in Mathematics, 724. Springer, Berlin,

\mn
Kadanoff, L.P., Coppersmith, S., and Aldana, M. (2002) Boolean dynamics with random couplings.
arXiv:nlin.AO/0204062

\mn
Kauffman, S. A. (1969) Metabolic stability and epigenesis in randomly constructed genetic nets. {\it Journal of Theoretical Biology}, 22, 437--467

\mn
Kauffman, S. A. (1993) {\it Origins of Order: Self-Organization and Selection in Evolution}. Oxford University Press.

\mn
Kauffman, S.A., Peterson, C., Samuelson, B., and Troein, C. (2003) Random Boolean models and the yeast transcriptional network.
{\it Proceedings of the National Academy of Sciences} 110, 14796--14799

\mn
Li, F., Long, T., Lu Y., Ouyang Q., Tang C. (2004) The yeast cell-cycle is robustly designed. {\it Proceedings of the National Academy of Sciences}, 101, 4781--4786.

\mn
Liggett, T. M.(1999) {\it Stochastic Interacting Systems: Contact, Voter, and Exclusion Processes}. Springer.

\mn
Nyter, M., Price, N.D., Aldana, M., Ramsey. S.A., Kauffman, S.A., Hood, L.E., Yli-Harja, O., and Shmuelivich, I. (2008)
{\it Proceedings of the National Academy of Sciences.} 105, 1897--1900

\mn
Shmulevih, I., Kauffmann, S.A., and Aldana, M. (2005) Eukaryotic cells are dynamically ordered or critical but not chaotic.
{\it Proceedings of the National Academy of Sciences.} 102, 13439--13444

\end{document}